\documentclass[10pt]{amsart}

\usepackage{indentfirst,graphicx,placeins}
\usepackage{cite,dsfont}
\usepackage{amsmath,amsthm,amssymb,amsfonts}
\usepackage{caption}
\usepackage{overpic}

\newtheorem{thm}{Theorem}[section]
\newtheorem{lem}[thm]{Lemma}

\newcommand{\lcm}{\operatorname{lcm}}

\begin{document}
\title[Groups with Strong Symmetric Genus up to 25]{Classification of Groups with Strong Symmetric Genus up to Twenty-Five}

\author[Fieldsteel]{Nathan Fieldsteel}
\email{fieldst2@illinois.edu}
\address{Department of Mathematics, University of Illinois at Urbana-Champaign,  1419 W. Green St, Urbana, IL 61801}

\author[Lindberg]{Tova Lindberg}
\email{tlindberg@math.arizona.edu}
\address{Department of Mathematics, The University of Arizona, Box 210089, Tucson, AZ 85719-0089}

\author[London]{Tyler London}
\email{tyler.a.london@gmail.com}

\author[Tran]{Holden Tran}
\email{holdentran2007@u.northwestern.edu}

\author[Xu]{Haokun Xu}
\email{xuh@math.ucla.edu}
\address{Department of Mathematics, University of California, Los Angeles, Box
951555, Los Angeles, CA 90095-1555}

\date{24 March 2011 
}

\begin{abstract}
The strong symmetric genus of a finite group $G$ is the minimum genus of a compact Riemann surface on which $G$ acts as a group of automorphisms preserving orientation. A characterization of the infinite number of groups with strong symmetric genus zero and one is well-known and the problem is finite for each strong symmetric genus greater than or equal to two. May and Zimmerman have published papers detailing the classification of all groups with strong symmetric genus two through four. Using the computer algebra system {\sf GAP}, we extend these classifications to all groups of strong symmetric genus up to twenty-five.  This paper outlines the approach used for the extension.
\end{abstract}

\maketitle

\section{Introduction}

The strong symmetric genus is a topological concept that can be explored using computational group theory.  Suppose $G$ is a finite group.  The \emph{strong symmetric genus} of $G$, denoted $\sigma^0(G)$, is the minimum genus of a compact Riemann surface on which $G$ acts faithfully as a group of automorphisms while preserving orientation.  An interesting question to ask is exactly which groups act on a compact Riemann surface of genus $g$ in this way, and on no surfaces of smaller genus.  These groups are said to have strong symmetric genus $g$.  

There are an infinite number of groups with $g=0$ and $g=1$, and a finite number of groups for all $g \geq 2$. Groups of strong symmetric genus zero and one are addressed in \cite[$\S6.3$]{gross} and \cite[$\S3.4$]{Breuer}. All groups of strong symmetric genus zero have the form $C_{n}$ or $D_{2n}$ for all positive integers $n$, $S_{4}$, $A_{4}$, or $A_{5}$.  Groups of strong symmetric genus one have no such clearly-defined classification, but they can all be obtained as a quotient of one of five particular Fuchsian groups, as will be discussed later.  In \cite{Zimm1}, May and Zimmermann classified all groups with strong symmetric genus two and three. Later, they extended this classification to all groups with strong symmetric genus four in \cite{Zimm2}. This work by May and Zimmerman was done essentially without the use of computer software.

Our research extends the classifications made by May and Zimmermann, using the computer algebra system {\sf GAP} \cite{GAP4}.  We have written a procedure using the {\sf GAP} Small Groups Library \cite{bes} that has allowed us to classify all groups of strong symmetric genus up to 25. The classification given by this {\sf GAP} procedure agrees with the previously known classifications for strong symmetric genus up through four.

\section{Preliminaries}
\subsection{Groups of Automorphisms of Compact Riemann Surfaces}
For our purposes, a Riemann surface will be the orbit space of a suitable group action. In particular, we will only consider the action of certain Fuchsian groups on the upper half plane $\mathbb{H}$ which force the Riemann surfaces in consideration to be compact. Recall that a\textit{Fuchsian surface group}, denoted by $K$ for the remainder of this paper, is a torsion-free Fuchsian group. The following result shows exactly when a finite group acts as a group of automorphisms of a compact Riemann surface.

\begin{thm}
\cite[5.9.5]{Jones} A finite group $G$ acts as a group of automorphisms of some compact Riemann surface of genus $g\geq2$, if and only if $G$ is isomorphic to $\Gamma/K$ where $\Gamma$ is a Fuchsian group with compact orbit space, and $K$ is a Fuchsian surface group with orbit genus $g$ (the genus of the Riemann surface $\mathbb{H}/K$) that is a normal subgroup of $\Gamma$.
\end{thm}

Now, as an abstract group, every Fuchsian group $\Gamma$ can be described by its signature (even up to isomorphism). A signature consists of a nonnegative integer called the orbit genus followed by a list of integers greater than or equal to $2$ called the periods. One result that our procedure uses for combinatorial purposes is the Riemann-Hurwitz formula, given in Theorem \ref{rhf} below. This formula relates the signature of $\Gamma$ with the index of a normal subgroup $K$ of $\Gamma$ and the orbit genus of $K$.

\begin{thm}\text{\cite[Chapter 1.4]{Breuer}}\label{rhf}
Suppose $\Gamma$ is a Fuchsian group with signature $(g;\, m_1, m_2,\ldots, m_r)$ and $K$ is a torsion-free normal subgroup of $\Gamma$ of finite index. Let $G$ be a finite group with $G\cong \Gamma/K$ and $X$ be the Riemann surface given by $\mathbb{H}/K$. Then the genus of $X$ and the orbit genus of $\Gamma$ are related by the equation $$g(X)-1=|G|\left(g\left(\Gamma\right)-1\right)+ \frac{|G|}{2}\sum_{i=1}^r \left(1-\frac{1}{m_i}\right).$$
\end{thm}

Suppose there is an epimorphism $\Phi:\Gamma\rightarrow G$. We say that $\Phi$ is a \emph{surface kernel epimorphism} if the kernel of $\Phi$ is a Fuchsian surface group. The following lemma gives a bound for the order of such a finite group $G$, supposing the existence of a surface kernel epimorphism from a Fuchsian group $\Gamma$ to $G$.

\begin{lem}\text{\cite[p. 16]{Breuer}}\label{skebig}
Suppose there exists a surface kernel epimorphism $\Phi$ from the group $\Gamma$ with signature $(g_0;\, m_1, m_2, \ldots, m_r)$ to a finite group $G$ such that the kernel of $\Phi$ is torsion-free and g$(\ker\Phi)=g \geq 2$.  Then we have the following:

\begin{itemize}
\item[(a)] If $g_0 > 0$ or $r \geq 5$, then $|G| \leq 4(g-1)$.
\item[(b)] If $r=4$, then $|G| \leq 12(g-1)$.
\item[(c)] If $|G| \geq 24(g-1)$, then $g_0=0$ and $r=3$, and $\Gamma$ has one of the following signatures in Table 1 with the corresponding group size for G.
\end{itemize}

\begin{table}[h!!!!!!!!!!!!!!!!!!!!!!!!!!!!!!]
{\small \label{sigsforlargeorders}
\begin{center}
\begin{tabular}{|c|p{0.6in}|}
\hline $(g_0;m_1, m_2, m_3)$  & \multicolumn{1}{c|}{$|G|$}\\ \hline
$(0;2,3,7)$     & $84(g-1)$\\
$(0;2,3,8)$     & $48(g-1)$\\
$(0;2,4,5)$     & $40(g-1)$\\
$(0;2,3,9)$     & $36(g-1)$\\
$(0;2,3,10)$    & $30(g-1)$\\
$(0;2,3,11)$    & $\frac{132}{5}(g-1)$\\
$(0;2,3,12)$    & $24(g-1)$\\
$(0;2,4,6)$     & $24(g-1)$\\
$(0;3,3,4)$     & $24(g-1)$\\
\hline
\end{tabular}
\end{center}
}
\begin{center}
\caption{Signatures for ``Large'' Group Orders}
\end{center}
\end{table}
\end{lem}

\subsection{{\sf GAP} Small Groups Library}
The main tool which our procedure uses for classifying the groups of a given strong symmetric genus is the computer algebra system {\sf GAP} (Groups, Algorithms and Programming). In particular, we utilize the {\sf GAP} Small Groups Library. The Small Groups Library allows access to all groups of ``small'' orders.  The  groups  are sorted by their orders and they are listed up to isomorphism;  that  is,  for  each  of the available orders a complete and
irredundant list of isomorphism type representatives of groups is given. Therefore, each group in the library can be identified by its {\sf GAP} Group ID, a pair of numbers $[s,k]$, where $s$ is the order of the group and $k$ is an integer between 1 and the number of groups of order $s$. 

\section{Description of Procedure}

For a given integer $g$, the purpose of our program is to determine all groups with strong symmetric genus $g$.  The program is written to perform computations in three sections: the first uses combinatorics to generate lists of possible groups and signatures, the second searches for the appropriate surface kernel epimorphisms, and the third checks to make sure that groups satisfying all necessary conditions have not appeared with smaller genus.  All of these computations are based on the input $g$; the output is a list of all groups with strong symmetric genus $g$.  Throughout this entire procedure, we process groups in terms of their {\sf GAP} group ID $[s,k]$.  Functions we employ are denoted by a \texttt{teletype} font.

\subsection{Generation of Signatures}

The first section of our code generates lists of signatures and finite groups.  This is done through combinatorial methods, established by the Riemann-Hurwitz formula [Theorem \ref{rhf}], which bounds the number of possible finite groups by bounding their possible size [Lemma \ref{skebig}].  Also from this formula, bounds on the number of periods in a given signature are established.  This bound follows from the fact that we are only interested in surface kernel epimorphisms: we can thus restrict the numbers that can occur in the period of a signature associated with a particular group order to those numbers which divide the group order.  These two facts first enable us to work with a finite list of possibilities and then help us streamline the process in this section of the code, cutting down the number of possibilities that must be tested for surface kernel epimorphisms.  The only groups and signatures which are passed out of this section of the code are those that satisfy the Riemann-Hurwitz formula for the given genus $g$.

\subsection{Existence of Surface Kernel Epimorphisms}

Once the list of candidates for $\Gamma$ and $G$ has been generated, the second section of our code either finds a surface kernel epimorphism $\Phi : \Gamma\rightarrow G$ or eliminates the possibility of one.  We employ the $\sf{GAP}$ library function \texttt{GQuotients} to find these surface kernel epimorphisms.  As this is a very time expensive function, we first seek to eliminate as many candidates as possible before entering this function.  Based on the structures of the groups $\Gamma$ and $G$ in consideration, we are able to remove certain groups by performing a series of tests.

The first of these tests is the function \texttt{abelianInvariantsCondition}, which eases the process for abelian $G$.  Given two groups $S$ and $T$, it performs an arithmetic check to determine whether or not any kind of epimorphism can exist from $S$ onto $T$.  Theorem \ref{abinvcond} is a more precise statement of the underlying theory of \texttt{abelianInvariantsCondition}, which reduces the problem of finding an epimorphism to a problem in terms of \emph{p}-groups.  In the theorem, we denote $S/\,[S,S]$ by $\Gamma$ and $T/\,[T,T]$ by $G$.

\begin{thm}\label{abinvcond}

If $\Gamma$ is a finitely-generated abelian group and $G$ is a finite abelian group, then $\Gamma \cong \mathbb{Z}^r \oplus H$ for some $r \geq 0$ and a finite group $H$, and there is an epimorphism $\Phi: \Gamma \rightarrow G$ if and only if, for each prime divisor $p$ of $|G|$, there is an epimorphism $\Phi_p: \mathbb{Z}^r \oplus H_p \rightarrow G_p$, where $H_p$ and $G_p$ are the unique Sylow p-subgroups of $H$ and $G$, respectively.

\end{thm}

\begin{proof}
Let $p$ be a prime dividing $|G|$ and suppose $\Phi: \Gamma \rightarrow G$ is an epimorphism.  Clearly $\Phi(H_p) \subseteq G_p$.  Let $\Pi_p: G \rightarrow G_p$ be the canonical projection onto $G_p$ and let $\Phi_p = \Pi_p \circ (\Phi |_{\mathbb{Z}^r \oplus H_p})$.  Then $\Phi_p$ is a homomorphism, and it is surjective because for any $x \in G_p$, we can take its pre-image $y$ under $\Phi$ which will be of the form $y = y_0 + \sum_{q \hspace{2pt} \mid \hspace{2pt} |H|}(y_q)$ where $y_0 \in \mathbb{Z}^r$ and $y_q\in{}H_q$ for each prime $q$ dividing $|H|$.  Then if $p$ divides $|H|$ we have $\Phi_p(y_0 +y_p) = x$, and otherwise we have $\Phi_p(y_0) = x$.

Now suppose there are epimorphisms $\Phi_p: \mathbb{Z}^r \oplus H_p \rightarrow G_p$ for all primes $p$ dividing $|G|$. Define a homomorphism $\Phi: \mathbb{Z}^r \oplus H \rightarrow G$ as follows: If $t$ is a generator of $\mathbb{Z}^r$, then $t$ is mapped to $\sum_p\Phi_p(x)$ by $\Phi{}$. If $t$ is a generator of $H_p$ for some $p$ dividing $|G|$,  then $t$ is mapped to $\Phi_p(x)$ by $\Phi{}$.

Then $\Phi$ is a homomorphism.  To show that it is surjective, it suffices to show that the generators of $G_p$ are in the image of $\Phi$ for all $p$ dividing $|G|$.  Let $x$ be a generator of $G_p$ for some $p$ and choose a pre-image $y$ of $x$ under $\Phi_p$.  Then $y=y_0+y_p$ with $y_0 \in \mathbb{Z}^r$, $y_p \in H_p$ and $\Phi(y) = \Phi_p(y_p) + \sum_q\Phi_q(y_0) = \Phi_p(y) + \sum_{q \neq p} \Phi_q(y_0)$.  The order of $\Phi_p(y)$ is a power of $p$, say $p^k$.  Then the order of $\sum_{q \neq p}\Phi_q(y_0)$, say $N$, is coprime to $p$.

Now choose $\tilde{N}$ with the properties $\tilde{N} \equiv 0$ (mod N) and $\tilde{N} \equiv 1$ (mod $p^k$).  Then $\Phi(\tilde{N} \cdot y) = \Phi_p(y) = x$.
\end{proof}

Another test we perform prior to calling \texttt{GQuotients} is checking whether groups on our list have strong symmetric genus $2 \leq \sigma^0 \leq g-1$, and removing from consideration those that do.  We have stored the data for groups of $2 \leq \sigma^0 \leq g-1$ from previous iterations of our procedure, so at this point, we discard groups with $2 \leq \sigma^0 \leq g-1$ by scanning through these databases.

Our third test is to handle the case when $G$ is abelian. We applied the following generalization of \cite[Theorem 4]{Harv}:
\begin{thm}\text{\cite[Theorem 9.1]{Breuer}}\label{abeliancheck}
Let $\Gamma$ be a Fuchsian group with signature $(g_0; m_1, \ldots, m_r)$ and let $M = \lcm(m_1, m_2,\ldots,m_r)$. There is a surface kernel epimorphism from $\Gamma$ onto a finite abelian group $G$ if and only if the following conditions are satisfied.
	\begin{itemize}
	\item[(o)] There exists an epimorphism from $\Gamma$ onto $G$;
	\item[(i)] $\lcm(m_1, m_2, \ldots, m_{i-1},m_{i+1}, \ldots, m_r) = M$ for all $i$;
	\item[(ii)] $M$ divides the exponent $\exp(G)$ of $G$, and if $g_0 = 0, M = \exp(G)$;
	\item[(iii)] $r\neq 1$ and if $g_0 = 0, r \geq 3$;
	\item[(iv)] If $M$ is even and only one of the abelian invariants of $G$ is divisible by the maximum power of $2$ dividing $M$, then the number of $m_i$ divisible by the maximum power of $2$ dividing $M$ is even. 
	\end{itemize}
\end{thm}
Observe that all groups which have passed \texttt{abelianInvariantsCondition} will automatically satisfy condition (o). Another possible method to handle this case is to compute the strong symmetric genus of $G$ under the restriction $g \geq 2$ using the formula from \cite[Theorem 4]{Mac}.

In order to eliminate more groups from our list, we again use the fact that a surface kernel epimorphism must preserve the orders of the generators. Let $G$ be the target group of a surface kernel epimorphism from a group with a signature of the form $(0;p, p, \ldots, p)$, where $p$ is some prime.  We compute the conjugacy classes of $G$ and select those classes which consist of elements of order $p$.  The normal closure of the selected conjugacy classes is calculated and if it is not equal to the entire group $G$, then the signature can be discarded.

At this point all remaining groups enter \texttt{SpecializedGQuotients} where a surface kernel epimorphism is constructed, if possible. \texttt{SpecializedGQuotients} is a modification of the {\sf GAP} Library function \texttt{GQuotients} which constructs all possible epimorphisms from one group onto another.  To enhance the efficiency of our code, we made two changes to \texttt{GQuotients}: first, we only search for surface kernel epimorphisms; that is, we send generators of a group to elements in the target group of the exact same order.  Second, we stop constructing surface kernel epimorphisms after one is found since we only require the existence of one.

\subsection{Checks for Groups of Strong Symmetric Genus Zero and One}

In order to ensure that $g$ is the strong symmetric genus of each group on our list, it must be the minimum genus of any surface on which a given group acts as a group of automorphisms.  We ensure that this is so by checking that each group on our list has not already occurred for smaller strong symmetric genus.  Recall that groups of strong symmetric genus $2\leq\sigma^0\leq g-1$ have already been eliminated.  For strong symmetric genus 0, we eliminate groups of the form $C_n$ and $D_{2n}$ for all $n$, $S_4$, $A_4$ and $A_5$.  A group $G$ has strong symmetric genus 1 if it is the image of a Fuchsian group represented by one of the following signatures: $(1; -)$, $(0; 2, 2, 2, 2)$, $(0; 3, 3, 3)$, $(0; 2, 4, 4)$ or $(0; 2, 3, 6)$ \cite[$\S 3.4$]{Breuer}.  It is possible that some $G$ on our list meets this criterion and therefore has strong symmetric genus 1.  We thus run \texttt{GQuotients} again, eliminating any groups for which this is the case.  This leaves a list of all groups of strong symmetric genus $g$.

\section{{\sf GAP} Results and Conclusions}

This procedure produces a classification of all groups of strong symmetric genus up to twenty-five.  The data, as well as the code, can be found at 
	\begin{center}
	\texttt{http://math.arizona.edu/$\sim$sp2008/}
	\end{center} 
Table~\ref{genusresults} below gives $\nu(g)$, the number of groups for each strong symmetric genus up to twenty-five. 

\begin{table}[h!!!!!!!!!!!!!!!]
{\small 
\begin{center}
\begin{tabular}{|r|l| r |r|l| r |r|l| r |r|l|}
\hline $g$  & $\nu(g)$ & & $g$ & $\nu(g)$ & & $g$ & $\nu(g)$ && $g$ & $\nu(g)$\\ \hline
$2$ & 6 	& \hbox{\hskip 8pt}& 
	$8$ & 12 	& 
		\hbox{\hskip 8pt}& $14$ & 14	& 
			\hbox{\hskip 8pt}& $20$ & 15\\ 
$3$ & 10 	& \hbox{\hskip 8pt}& $9$ & 55 	& \hbox{\hskip 8pt}& $15$ & 31	& \hbox{\hskip 8pt}& $21$ & 86\\
$4$ & 10 	& \hbox{\hskip 8pt}& $10$ & 36 	& \hbox{\hskip 8pt}& $16$ & 23	& \hbox{\hskip 8pt}& $22$ & 28\\
$5$ & 22 	& \hbox{\hskip 8pt}& $11$ & 29 	& \hbox{\hskip 8pt}& $17$ & 129	 &\hbox{\hskip 8pt}& $23$ & 23\\
$6$ & 11 	& \hbox{\hskip 8pt}& $12$ & 14 	& \hbox{\hskip 8pt}& $18$ & 16	& \hbox{\hskip 8pt}& $24$ & 17\\
$7$ & 22 	& \hbox{\hskip 8pt}& $13$ & 50 	& \hbox{\hskip 8pt}& $19$ & 62	& \hbox{\hskip 8pt}& $25$ & 167\\
\hline
\end{tabular}
\end{center}
}
\caption{Number of Groups $\nu(g)$ for Each Strong Symmetric Genus $g$}
\label{genusresults}
\end{table}

\section{Comments and Open Problems}

The obvious open problem is to continue this classification further. The reason our data ends for strong symmetric genus 25 is that our code depends on the {\sf GAP} Small Groups Library, which does not contain a full classification of groups of order 2016. As Lemma \ref{skebig} demonstrates, groups of order 2016 are first encountered when looking at surfaces of genus 25. This Lemma also states that such groups \emph{must} arise from the signature $(0; 2, 3, 7)$. It is well-known that every finite nontrivial quotient of the triangle group $(0; 2, 3, 7)$ is a Hurwitz group \cite{cond}. Thus, if there were a surface kernel epimorphism from this group to a group $G$ of order 2016, then $G$ must be a Hurwitz group, and thus a perfect group \cite{cond}. Using the {\sf GAP} Perfect Groups Library, it can be checked that there exist no perfect groups of order 2016. Thus, this signature and group can be eliminated from the possibilities for strong symmetric genus 25, and our classification of groups with strong symmetric genus 25 is complete.

The classification could be continued by incorporating this method into our code. All Hurwitz groups of order less than $10^{6}$ have been classified \cite{cond}, and this along with the {\sf GAP} Small Groups Library could most likely classify all groups with strong symmetric genus up to 42 (since $48\cdot41<2000)$. Using the {\sf GAP} Perfect Groups Library, one could continue the classification through strong symmetric genus 60.

In \cite[$\S7$]{Zimm2}, May and Zimmerman introduced a number $\nu(g)$ which is the number of groups with strong symmetric genus $g$.  They have already computed $\nu(g)$ for $g$ up to four.  In Figure \ref{fig:graph}, we plot the values of $\nu(g)$ for $g$ up to twenty-five. For all $g\geq 2$, Hurwitz's Bound [Lemma \ref{skebig}] guarantees that $\nu(g)$ is finite. Furthermore, for all $g$, $\nu(g)$ is at least one \cite{Zimm3}.

\begin{figure}[h]
\centering
	\begin{overpic}[scale=.6]
	{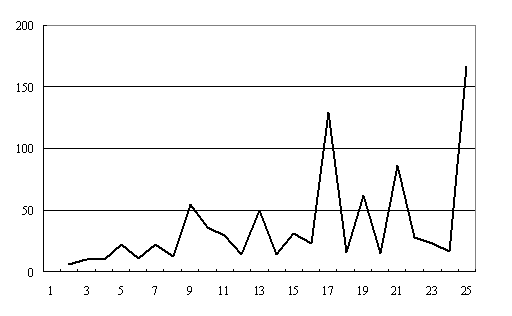}
		\put(-5,30){$\nu(g)$}
		\end{overpic}
		\caption{Plot of $\nu(g)$ for $2\leq g\leq 25$.}
		\label{fig:graph}
\end{figure}	

Notice in Figure \ref{fig:graph} that the number of groups for strong symmetric genera of the form $2^n+1$ for $n \in \mathbb{N}$ is higher than that for all preceding strong symmetric genera.  For example, there are 55 groups of strong symmetric genus nine.  The number of groups for all previous strong symmetric genera is significantly less than 55.  We believe this trend will continue for higher values of $2^n+1$.

\section{Acknowledgements} 

This paper arose from the 2008 Arizona Summer Program, which is sponsored by The National Science Foundation (EMSW21-VIGRE Award 0602173) and the University of Arizona.  The project was advised by Dr. Thomas Breuer, and additional thanks are also due to program director Klaus Lux and to Alexander Hulpke. Finally, we thank the referee for careful reading and helpful comments.

\bibliographystyle{alpha}
\bibliography{articlessgbib}

\end{document}